\documentclass{article}
\usepackage{amsfonts}
\usepackage{amsmath}
\usepackage{amssymb}
\usepackage[a4paper]{geometry}
\usepackage{mathrsfs}

\usepackage{amsmath,amssymb,amsthm,mathrsfs}
\usepackage[english]{babel}
\usepackage{xcolor}
\usepackage{enumerate}
\usepackage[colorlinks]{hyperref}
\renewcommand\eqref[1]{(\ref{#1})} %Need with hyperref

\newcommand{\C}{{\mathbb C}}
\newcommand{\N}{{\mathbb N}}
\newcommand{\R}{{\mathbb R}}
\newcommand{\Rn}{\mathbb{R}^n}
\newcommand{\h}{{\mathbb H}}
\newcommand{\slp}{{\mathcal L}}
\newcommand{\HS}{{\mathtt{HS}}}
\newcommand{\Gh}{{\widehat{G}}}

\DeclareMathOperator{\Dom}{Dom}
\DeclareMathOperator{\OP}{op}

\numberwithin{equation}{section}
\theoremstyle{plain}
\newtheorem{theorem}{Theorem}[section]
\newtheorem{proposition}[theorem]{Proposition}
\newtheorem{cor}[theorem]{Corollary}

\theoremstyle{definition}
\newtheorem{definition}[theorem]{Definition}
\newtheorem{remark}[theorem]{Remark}
\newtheorem{ex}[theorem]{Example}
\newtheorem{convention}[theorem]{Convention}

\AtEndDocument{\bigskip{\footnotesize%
		\textsc{V\'eronique Fischer, Department of Mathematical Sciences, University of Bath, Great Britain} \par  
		\textit{E-mail address}, V. Fischer:: \texttt{v.c.m.fischer@bath.ac.uk} \par
		\addvspace{\medskipamount}
		\textsc{Michael Ruzhansky, Department of Mathematics, Imperial College London, Great Britain}\par
		\textit{E-mail address}, M. Ruzhansky: \texttt{m.ruzhansky@imperial.ac.uk}\par
		\addvspace{\medskipamount}
		\textsc{Chiara Alba  Taranto, Department of Mathematical Sciences, University of Bath, Great Britain}\par
		\textit{E-mail address}, C. A. Taranto: \texttt{ct763@bath.ac.uk}		
}}
\begin{document}

\title{Subelliptic Gevrey spaces}

\author{V\'eronique Fischer
	\and
	Michael Ruzhansky
	\and 
	Chiara Alba Taranto}

\maketitle

\begin{abstract} 
	
	In this paper, we define and study Gevrey spaces associated with a H\"ormander family of (globally defined) vector fields and its corresponding sub-Laplacian.
	We show some natural relations between the various Gevrey spaces in this setting on general manifolds, and more particular properties on Lie groups with polynomial growth of the volume. In the case of the Heisenberg group and of $SU(2)$, we show that all our descriptions coincide.	
	
\end{abstract}

\renewcommand{\thefootnote}{\fnsymbol{footnote}} 
\renewcommand{\thefootnote}{\arabic{footnote}} 

\tableofcontents

\section{Introduction}	

In 1918 the French mathematician Maurice Gevrey introduced in \cite{G1918} the `\textit{fonctions de classe donn\'ee}', later called \textit{Gevrey functions} in his honour:
\begin{definition}[Gevrey functions of order $s$ in $\Omega$]
	Let $\Omega$ be an open subset of $\Rn$ and let $s\geq 1$. A function $f$ is a Gevrey function of order $s$, written $f\in G^s(\Omega)$, when $f\in\mathcal{C}^\infty(\Omega)$ and for every compact subset $K$ of $\Omega$ there exist two positive constants $A$ and $C$ such that for all $\alpha\in\N^n_0$ and for all $x\in K$ we have
	\[
	|\partial^\alpha f(x)|\leq A C^{|\alpha|}(\alpha!)^s.
	\]
\end{definition}

It follows immediately from the definition that for $s=1$ the corresponding Gevrey class of functions coincides with the space of real analytic functions, while in general they provide an intermediate scale of spaces between smooth functions $\mathcal{C}^\infty$ and real-analytic functions. This means that Gevrey classes are widely relevant in the analysis of operators with  some properties failing in $\mathcal{C}^\infty$ or in the analytic frameworks.

A simple but meaningful example is the homogeneous equation associated to the heat operator $L=\partial_{t}-\sum_{j=1}^{n}\partial_{x_j}^2$ in $\Rn$ with $n\geq 1$. Indeed, the solutions of the homogeneous equation $Lu=0$ are not analytic in general, though always $\mathcal C^\infty$, and by calculating derivatives of the fundamental solution of $L$ we can deduce they are Gevrey for $s\geq 2$. This provides more precise information on the regularity of the solutions of the heat equation. An example in the other direction is that the Cauchy problem for the wave equation is analytically well-posed but not well posed in $\mathcal C^\infty$ in the presence of multiple characteristics. Consequently, determining the sharp Gevrey order for the well-posedness is a challenging problem with several results, starting with the seminal work of Colombini, de Giorgi and Spagnolo \cite{CDS79}, and continuing with many others, such as \cite{CJS1987,CS1982}.

\medskip

This paper proposes to study the equivalence between possible definitions of Gevrey spaces.
We will restrict ourselves to spaces of compactly supported functions.
First let us recall some well-known equivalent characterisation of the compact Gevrey spaces for $s>0$
(from e.g. \cite{R1993}).
For a smooth function $\phi:\R^n\to \C$ with compact  support with compact support, the following are equivalent:
\begin{itemize}
		\item[(i)]  there exist $A,C>0$ such that for every $\alpha\in\N^n$ we have $\|\partial^\alpha\phi\|_{L^\infty}\leq CA^{|\alpha|}(\alpha!)^s$;
		\item[(ii)] 
		 there exist $A,C>0$ such that for every $\alpha\in\N^n$ we have $\|\partial^\alpha\phi\|_{L^2}\leq CA^{|\alpha|}(\alpha!)^s$;
		\item[(iii)]  there exist $A,C>0$ such that for every $k\in \N_{0}$ we have $\|\Delta^{k}\phi\|_{L^2}\leq CA^{2k}((2k)!)^s$,\\
		where $\Delta$ denotes the standard Laplacian on $\R^{n}$.
	\end{itemize}
	This defines the Gevrey space for compactly supported functions on $\R^n$, which we denote by $G_0^s(\R^n)$ or just  $G_0^s$.	
	
The argument of this paper is to define and study Gevrey spaces associated with a family of vector fields $\mathbf X=\{X_{1},\ldots, X_{r}\}$ on a general (smooth) manifold $M$ and with its associated sub-Laplacian $\slp =- \sum_{j}X_{j}^{2}$.
We may say that the Gevrey spaces $G_0^{s}$  of $\R^{n}$ are \emph{Euclidean} whereas the ones defined in association with $\mathbf X$ and $\slp$ are \emph{subelliptic} when $\mathbf X$ satisfies the H\"ormander condition, or even \emph{elliptic} when the corresponding operator $\slp$ is elliptic.
The definitions we put forwards rely on the observation that the characterisations (i), (ii) and (iii) above also make sense on any manifold $M$ equipped with a measure, replacing the derivative $\partial^{\alpha}$ with any product of $|\alpha|$ vector fields in $\mathbf X$ in a given order (see Convention \ref{conventionI}). 
In this general context under some natural assumptions, 
Properties (i) and (ii) are equivalent,
and  Property (ii) implies Property (iii). 
However, the implication (iii) $\implies$ (ii) proves more challenging,
and here we prove it only in certain cases, namely for the compact group $SU(2)$ and for the Heisenberg group $\mathbb H_{n}$.
In the elliptic case, this was proved in \cite{DR2014} for compact Lie groups,  
and we give here a proof with more direct arguments than in \cite{DR2014}.
The difficulty in showing the equivalences between these possible definitions from an analytic viewpoint comes from the fact that the vector fields in $\mathbf X$ may not  commute - unlike in the Euclidean case.

\medskip

We want to emphasise the fact that this paper concerns global properties of objects globally defined. 
In particular, the question we pose in this paper concerns a H\"ormander family of vector fields $\mathbf X$ which are \emph{globally} defined on $M$.
This means considering only certain settings since  general Riemannian or sub-Riemannian  manifolds do not always possess a global orthonormal frame (for instance, for a Riemannian surface, this is equivalent to the manifold being orientable and the existence of a nowhere vanishing vector field). 
Furthermore, 
this paper is not concerned with defining sub-Riemannian Gevrey spaces globally from  properties like (i)  in local frames. Indeed, this would require  understanding which types of subRiemannian Gevrey or analytical properties on the atlas of the  manifold and on the vector fields would allow for a suitable global definition, and these investigations are not part of the objectives of this paper.
Note however that the concept of Gevrey manifolds is well defined using Euclidean Gevrey spaces.
Furthermore,  it is possible to define Gevrey spaces from the analogue of (iii) replacing $\Delta$ with the Laplacian or sub-Laplacian  on a general Riemannian or sub-Riemannian manifold;
this requires fixing a measure in the sub-Riemannian case  while the natural choice in the Riemannian case is  the volume element.

\medskip

On $\R^n$, 
(iii)  is equivalent to (see \cite{R1993}):
\begin{itemize}
		\item[(iii)']  there exists $D>0$ such that $\|e^{D\Delta^{\frac{1}{2s}}}\phi\|_{L^2}<\infty$.
			\end{itemize}
We observe that the analogues of (iii) and (iii)' would make sense replacing $\Delta$ with $\slp$ in our context or with a Laplacian  and or a sub-Laplacian on a Riemannian or sub-Riemannian manifold.
All these operators are essentially self-adjoint on $L^2(M)$, and 
their functional analysis implies that the analogues of (iii) and (iii)'  are in fact equivalent.
This equivalence   has been studied  in abstract contexts for a long time, see for instance the pioneering works of Nelson \cite{N1959} and of Chernoff \cite{Chernoff1975} 
on analytic vectors
and  quasi-analytic vectors respectively.

The Euclidean Gevrey spaces can be effectively characterised on the Fourier transform side. Indeed,  Property (iii) on $\R^{n }$ can also be shown \cite{R1993} to be equivalent to 
\begin{itemize}
		\item[(iii)''] there exists $D',K>0$ such that
		$|\widehat \phi(\xi)|\leq K e^{-D' |\xi|^{1/s}}$ holds for all $\xi\in \R^{n}$.
			\end{itemize}
The Fourier transform and therefore this last property do not make sense on a general manifold.
It is a very handy characterisation of Euclidean Gevrey membership and it makes sense on Lie groups. 
In \cite{DR2014} Dasgupta and the second author showed the equivalence between the analogues of (i), (ii), (iii), (iii)' and (iii)'' on a compact Lie group. 
In fact, some of the characterisations in \cite{DR2014} 
were already known from the Gevrey spaces introduced in relations with representations theory in 
\cite[Section 1]{GW1980},
see especially Lemma 1.8 therein.
In this paper (see Section \ref{SSEC:laplacian}), 
we also recover  the main results of \cite{DR2014} with more direct arguments.

Characterisation such as in \cite{DR2014} was used in \cite{GR2015} to find energy estimates for the corresponding wave equations for the Laplacian and establish a well-posedness result in Gevrey classes. 
This can be viewed as an extension of the work by Seeley \cite{S1969} 
where analytic functions on compact manifolds were characterised in terms of their eigenfunction expansions. 
Subsequently, Dasgupta and the second author studied the case of compact manifolds for an elliptic operator \cite{DR2016}.
The characterisation of Euclidean Gevrey spaces on the Fourier side 
 is particularly relevant for applications (see e.g. \cite{DFT2009}), most notably allowing one to obtain energy estimates for evolution partial differential equations \cite{R1993} as well as  for the well-posedness questions for hyperbolic PDE's such as in \cite{CDS79}. 
 The latter questions were also studied by the second and third authors 
 in a subelliptic context \cite{RT2017}.

\medskip

Our proofs of the  relations between the subelliptic Gevrey spaces on an arbitrary manifold defined via Properties (i), (ii) and (iii)
rely on general functional analysis.
However, the equivalences between these three properties 
are proved in this paper only when the manifold $M$ is the following Lie group: 
firstly on any compact Lie group $G$ (for the elliptic case), 
secondly
on the compact group $G=SU(2)$ and thirdly on the Heisenberg group $G=\mathbb H_{n}$.
In each case, our result follows from obtaining bounds for the operator norm of the higher order Riesz transforms $R_\alpha:= \mathbf X_I  \slp^{-\frac {|I|}2}$ on $L^2(G)$
in terms of $|I|$.
It is known that these Riesz transforms are bounded on $L^2(G)$ (see \cite{ER1999} and Section \ref{REM:oppImplication}) and that, in the first case, that is, with $\slp$ being the Laplace--Beltrami operator on a compact Lie group $G$, all these norms are bounded by $1$. In the other two cases, we obtain an estimate as a constant to the power given by the length of the derivative $\mathbf X_I$, see Propositions \ref{PR:SU2Riesz} and \ref{PR:HnRiesz} respectively for $SU(2)$ and $\mathbb{H}_n$. These results are new to our knowledge and they are of interest in their own right. Our proofs rely on explicit expressions on the group Fourier side.

\medskip

Our present investigations suggest possible avenues of future research:
\begin{enumerate}
\item 
Is it possible to describe consistently properties such as (i) above on any chart of a suitable atlas of some interesting classes of manifolds? 

Naturally, this is true in the Euclidean case and gives rise to the notion of Gevrey manifold.

\item 
Is it possible to determine the settings for which the equivalences between the analogues of (i), (ii) and (iii) above hold?
This means considering a manifold $M$ either equipped with a global H\"ormander family $\mathbf X$ as in this paper or where point  1. just above is possible.   

In this paper, we give a positive answer for the elliptic Gevrey spaces on any compact groups, and for the subelliptic Gevrey spaces on $SU(2)$ and on the the Heisenberg group $\mathbf H_n$ (for any $n\in \N$).
We expect that our methods can be generalised to the groups of   Heisenberg types. Studying the extension to stratified nilpotent Lie groups of step 2 may give some insight into question 2. for sub-elliptic Gevrey spaces. While the case of step 2 may still broadly use our methods as explicit formulae close to the ones we use are known \cite{CRS2005,BFG2016},  we suspect that the investigations on  stratified nilpotent Lie groups of any step will require new tools.

Focussing on the case of stratified nilpotent Lie groups at first is natural as it gives a rich setting where the H\"ormander family $\mathbf X$ and the associated family are well understood. Furthermore, it is possible that the Rothschild--Stein machinery - also known as nilpotentisation - may help transfer results from the nilpotent case to the sub-Riemannian manifold setting.

\item 
Is it possible to relate sub-elliptic Gevrey spaces with other spaces? 
For instance, can they be related to elliptic Gevrey spaces?
In the case $s=1$, the Euclidean Gevrey spaces are exactly the space of analytical functions and therefore do not contain any compactly supported functions. Can it be the same for sub-elliptic Gevrey spaces? 
\end{enumerate}

We end this introduction with a word on our motivations to define these new functional spaces and  on their applications. 
Beside the ones already mentioned for wave equations above, 
a long-term goal would be to study the Gevrey hypoellipticity of a sum of square of vector fields on an open set of the Euclidean space or of sub-Laplacians on sub-Riemannian manifolds.
This is still an open question with many contributions for which we will cite \cite{BM2016, T2011} and refer to all the references therein;
note that Gevrey hypoellipticity and solvability of linear partial differential operators in the Euclidean context or on Gevrey manifolds is quite well understood \cite{ACG}.
It is not difficult to see that a sub-Laplacian $\slp$ will be hypoelliptic for the subelliptic Gevrey spaces associated with $\slp$, see Remark \ref{REM:hypoel}. So the problem has now shifted to studying the relations of the subelliptic Gevrey spaces with elliptic Gevrey spaces (see Question 3. above).

\medskip

The organisation of this paper is as follows. In Section \ref{SEC:mfds}, we introduce our suggested definitions for subelliptic Gevrey spaces on manifolds, inspired by Properties (i), (ii) and (iii) for the Euclidean Gevrey spaces. We prove some first relations among those spaces. Then in Section \ref{SEC:cpt-groups}, we consider the case of compact Lie groups. In particular we present the characterisation of elliptic Gevrey spaces on any compact group and subelliptic Gevrey spaces on the special unitary group $SU(2)$. Finally in Section \ref{SEC:heis}, we provide the more detailed description of subelliptic Gevrey spaces in the setting of the Heisenberg group $\h_n$. 

\begin{convention}
	In order to avoid confusion, in this section we clarify some notation and straightforward inequalities we will use in this paper.
\begin{itemize}
	\item $\mathcal C^\infty_0$ is the space of smooth functions with compact support;
	\item a measure defined on a smooth manifold is always assumed to be Borel and regular;
	\item the functions considered are always supposed to be measurable;
	\item for every $s>0$ the \textit{Gamma function} at $s$ is defined to be
	\begin{align}\label{gammafunction}
	\Gamma(s):=\int_{0}^\infty t^{s-1}e^{-t}dt;
	\end{align}	
\item for $n\in\N$, we have $\Gamma(n)=(n-1)!$.
In the rest of this paper, we may abuse the notation and write $x!$ for $\Gamma(x-1)$ for any $x>1$;
	\item 
if $\alpha=(\alpha_{1},\ldots,\alpha_{n})\in\N^n$ is a multi-index, then 
we define $|\alpha|=\alpha_{1}+\ldots + \alpha_{n}$ and
$$
\alpha!:=\alpha_1!\ldots\alpha_n! \ ;
$$ 
The following inequalities hold for any $\alpha\in \N^{n}$ and $\beta,\gamma\in \N$:
\begin{equation}
\label{ineq_fac}
\alpha!\leq|\alpha|!\leq n^{|\alpha|}\alpha!
\qquad\mbox{and}\qquad
(\beta + \gamma)!\leq2^{\beta+\gamma}\beta! \gamma! .
\end{equation}
\end{itemize}
\end{convention}

\section{Subelliptic Gevrey spaces on manifolds}
\label{SEC:mfds}

In this section, we propose three different kinds of definition for Gevrey type classes of functions on a smooth manifold $M$.
They are similar to  Properties (i), (ii) and (iii) given in the introduction for an open set $M=\Omega$ of $\R^n$.

We assume that we are given a family $\mathbf{X} =\{X_1,\dots,X_r\}$ of vector fields on $M$ which satisfies the H\"{o}rmander condition, that is, the vector fields are real-valued and at every point $x$ of $M$ the (real) Lie algebra generated by $\mathbf {X}$ coincides with the tangent space $T_x(M)$ at $x$. We can then define the corresponding (positive) sub-Laplacian operator to be
\[
\slp:=-(X_1^2+\dots+X_r^2).
\]
In view of the H\"ormander theorem \cite{H1967},  this is a positive hypoelliptic operator. Such families of vector fields and their properties have been extensively studied in the literature: see, e.g., \cite{B2014} and the references therein.

In Section \ref{subsec:def} we will propose definitions for Gevrey type spaces of functions in terms of the H\"ormander system $\mathbf X$ and also in terms of the sub-Laplacian  $\slp$.
We will show some inclusions between these spaces in a general setting and also on certain classes of Lie groups, see  Sections \ref{subsec:inclusion} and \ref{REM:oppImplication} respectively.

We will need the following conventions in order to replace the use of the derivatives $\partial^\alpha$ in $\R^n$. 
\begin{convention}
\label{conventionI}
Let $a$ be any positive integer. 
We denote by $T_a^r$ the set of  $a$-tuples  of integers in $\{1,\ldots,r\}$, that is,
 $$
 T_a^r: =\{ I=(i_1,\ldots, i_a) \ : \ 1\leq i_1,\ldots, i_a\leq r\}.
 $$
Furthermore, for each  $I=(i_1,\ldots, i_a) \in T_a^r$, we define the multi-index $\alpha_I\in \N_0^r$ whose  $j$th component, $j=1,\ldots,r$, is the number $\alpha_{I,j}$ of indices $i_1,\ldots, i_s$  equal to $j$.
We define the length and the factorial of $I$ to be the length and the factorial of the multi-index $\alpha_I$:
$$
|I|:=|\alpha_I|= \alpha_{I,1}+\ldots+\alpha_{I,r}= i_1+\ldots+i_s, 
\quad \mbox{and}\quad 
I! = \alpha_I! = \alpha_{I,1}!\ldots \alpha_{I,r}.
$$ 
Finally, we set
$$
{\mathbf X}_I := X_{i_1}\ldots X_{i_s}.
$$
If $a=0$ then $T_a^r:=\{\emptyset\}$ consists of the empty set and the corresponding operator is ${\mathbf X}_\emptyset =I$ the identity. 
\end{convention}
With this convention, if all the vector fields in $\mathbf X$ commute then 
${\mathbf X}_I = X_1^{\alpha_{I,1}}\ldots X_r^{\alpha_{I,r}}$.

\subsection{Definitions on manifolds}
\label{subsec:def}

 Let us start with our proposed definitions of Gevrey spaces associated with the H\"ormander system $\mathbf X$.

 This is the analogue of Property (i) in the introduction where $\partial^\alpha$ has been replaced with the non-commutative derivative ${\mathbf X}_I$ (see Convention \ref{conventionI}):

\begin{definition}[$(\mathbf{X},L^\infty)$-Gevrey spaces]\label{DEF:gsl}
	Let $s>0$. The Gevrey space $\gamma^s_{\mathbf{X},L^\infty}(M)$ associated with $\mathbf{X}$ of order $s$ on $M$ is the space of all  functions $\phi\in\mathcal{C}_0^\infty(M)$ satisfying 
	\begin{equation}\label{gevrey}
	\exists A,C>0\qquad \forall a \in \N_0, \qquad \forall I\in T_a^r\qquad 
	|{\mathbf X}_I \phi(x)|\leq C A^{|I|}(I!)^s,\;\textrm{ for all } x\in M.
	\end{equation}
\end{definition}

When the manifold $M$ is equipped with a measure $\mu$, 
we can also define the Gevrey spaces associated 
with the family $\mathbf X$ and
with the sub-Laplacian operator $\slp$ in a way analogous to Properties (ii) and (iii) in the introduction.

\begin{definition}[$(\mathbf{X},L^{2})$-Gevrey spaces]\label{DEF:gslL2}
	Let $s>0$. The \textit{Gevrey space $\gamma^s_{\mathbf{X},L^2}(M)$ associated with $\mathbf{X}$ with respect to the $L^2$-norm} of order $s$ on $M$  is the space of all functions $\phi\in L^2(M)$ satisfying
	\begin{equation}\label{gevrey-L2}
	\exists A,C>0\qquad \forall a \in \N_0, \qquad \forall I\in T_a^r\qquad 
	\|{\mathbf X}_I \phi\|_{L^2(M)}\leq C A^{|I|}(I!)^s.
	\end{equation}
\end{definition}

\begin{definition}[$\slp$-Gevrey spaces] \label{DEF:gs_slp}
	Let $s>0$. The \textit{Gevrey space} $\gamma^s_{\slp}(M)$ associated with the sub-Laplacian $\slp$ of order $s$ on $M$ is the space of all functions $\phi\in L^2(M)$ satisfying 
	\begin{equation}
	\label{eq:def:gs_slp}
	\exists A,C>0\qquad \forall b \in \N_0,\qquad
	\|\slp^b  \phi \|_{L^2(M)} \leq C A^{2b} ((2b)!)^{s}.
	\end{equation}
\end{definition}
\begin{remark}\label{REM:hypoel}
	From Definition \ref{DEF:gs_slp}, it follows that
	$\slp$ is hypoelliptic for $\gamma_{\slp}^s (M)$ in the sense that for every $f \in L^2 (M)$
	we have 
	\begin{align*}
	\slp f \in \gamma_{\mathbf{X},L^2}^s \implies f \in \gamma_{\mathbf{X},L^2}^s.
	\end{align*}
\end{remark}

We will also need to define a setting where Sobolev inequalities adapted to $\mathbf{X}$ holds. Here, as above, $M$ is a manifold equipped with a measure $\mu$.

\begin{definition}[Sobolev embedding]\label{def:Sob}
We say that the Sobolev embedding holds on $M$ for $\mathbf{X}$ with index $k\in \N$
when 
if $f\in L^{2}(M)$ is a function such that 
all the $L^{2}$-norms $\|{\mathbf X}_I f\|_{L^2}$
with $I\in T_a^r$, $a\leq k$ and ${\mathbf X}_I$ as in Convention \ref{conventionI}, are finite, then $f$ is continuous and bounded on $M$ and we have
$$
	\|f\|_{L^\infty(M)}\leq C\sum_{a\leq k}\sum_{I\in T_a^r} \|{\mathbf X}_I f\|_{L^2(M)},
$$
for some constant $C>0$ independent of $f$.
\end{definition}

\subsection{First Properties}
\label{subsec:inclusion}

Here we analyse the relations between the spaces defined in Section \ref{subsec:def}.

We start with examining the link between the $L^{\infty}$ and $L^{2}$ Gevrey spaces associated with $\mathbf X$:

\begin{proposition}[Equivalence between $L^\infty$-norm and $L^2$-norm]
	\label{PROP:prop1}
	Let $M$ be a smooth manifold equipped with a measure $\mu$.
	Let $\mathbf{X} =\{X_1,\dots,X_r\}$ be a H\"ormander system on $M$. 
\begin{enumerate}
\item $\gamma^s_{\mathbf{X},L^{\infty}}(M)\subset  \gamma^s_{\mathbf{X}, L^2}(M)$.
\item When the Sobolev embedding holds on $M$ for $\mathbf X$ (in the sense of Definition \ref{def:Sob}), then any $\phi\in \gamma^s_{\mathbf{X}, L^2}(M)$
is smooth; furthermore if $\phi$ has compact support then  $\phi\in \gamma^s_{\mathbf{X},L^{\infty}}(M)$.
\end{enumerate}
\end{proposition}

\begin{proof}
Part 1 follows from the embedding $L^\infty\hookrightarrow L^2$ for compactly supported functions, so it remains to show Part 2.
So let $\phi\in\gamma^s_{\mathbf{X}, L^2}(M)$.
For every $I\in T_a^r$, 
the Sobolev embedding applied to $f=\mathbf X_I\phi$ implies that $\mathbf X_I\phi$ is continuous and bounded by
$$
	\|\mathbf X_I\phi\|_{L^\infty}\leq C \sum_{|\beta|\leq k}\|\partial^{\beta} \mathbf X_I\phi\|_{L^2}\leq C \sum_{|\beta|\leq k}C_{\phi} A^{|\alpha+\beta|}\big( (\alpha+\beta)!\big)^s \leq C' A'^{|\alpha|}(\alpha!)^s,
$$
having used  the factorial inequalities in \eqref{ineq_fac}. This shows that $\phi\in  \gamma^s_{\mathbf{X},L^{\infty}}(M)$.
\end{proof}

There is a natural inclusion between the two kinds of $L^{2}$-Gevrey spaces defined in Definitions \ref{DEF:gs_slp} and \ref{DEF:gslL2}:

\begin{proposition}\label{PROP:prop3}
	Let $M$ be a manifold equipped with a measure $\mu$. Let $\mathbf{X} =\{X_1,\dots,X_r\}$ be a H\"ormander system and $\slp$ the associated sub-Laplacian $\slp=-\sum_{j=1}^rX_j^2$. 
Then 
$$
\gamma^s_{\mathbf{X},L^2}(M) \subset \gamma^s_{\slp}(M) .
$$
\end{proposition}

\begin{proof}
Recall  the multinomial theorem adapted to  elements of a non-commutative algebra:
	\[
	(Y_1+\dots+Y_m)^h=\frac{1}{m!}\sum_{k_1+\dots+k_m=h}\frac{h!}{k_1!\dots k_m!}\sum_{\sigma\in \text{sym}(m)}\prod_{1\leq t \leq m}Y_{\sigma(t)}^{k_t}.
	\]
Applying this to $Y_{j}=X_{j}^{2}$ and
 using the factorial inequalities  in \eqref{ineq_fac}, we obtain
	\begin{align*}
	\|\slp^k\phi\|_{L^2}&\leq\frac{1}{r!}\sum_{|\alpha|=k}\frac{k!}{\alpha!}\sum_{\sigma\in\text{sym}(r)}\big{\|}\underbrace{ (X_{\sigma(1)}^2)^{\alpha_1}\dots(X_{\sigma(r)}^2)^{\alpha_r} }_{\partial^{2\alpha}} \phi\big{\|}_{L^2} \leq C\sum_{|\alpha|=k}\frac{k!}{\alpha!}A^{2|\alpha|}((2\alpha)!)^s\\%\leq CA^{2k}((2k)!)^s\sum_{|\alpha|=k}\frac{k!}{\alpha!} \\
	&\leq CA^{2k}((2k)!)^s\sum_{|\alpha|=k}\frac{k!}{|\alpha|!}r^{|\alpha|}\leq CA^{2k}((2k)!)^sr^kk^{r-1}\leq CA'^{2k}((2k)!)^s,
	\end{align*}  
	with $A'=Ar$, and where we used that $\frac{k^{r-1}}{r^k}\leq 1$.
\end{proof}

We would like to prove the reverse inclusion to the one given  in Proposition \ref{PROP:prop3}. 
Before doing so in special contexts, we observe that,  under a natural assumption, 
the Gevrey spaces associated with a sub-Laplacian $\slp$ admits an equivalent description in terms of the exponential of the fractional power of $\slp$.
This description  is the analogue to Property (iii)' in the Euclidean setting, see the introduction.
And the natural assumption is that the sub-Laplacian is a essentially self-adjoint operator so that
it admits a functional calculus:

\begin{proposition}\label{PROP:prop2}
	Let $M$ be a manifold equipped with a  measure $\mu$
	and a H\"ormander system $\mathbf{X} =\{X_1,\dots,X_r\}$ of vector fields. We assume that associated sub-Laplacian $\slp=-\sum_{j=1}^rX_j^2$ is a  essentially self-adjoint operator on $L^2(M)$. 

A function  $\phi \in L^{2}(M)$ belongs to $\gamma^s_\slp(M)$ 
if and only if $\|e^{D\slp^{1/2s}}\phi\|_{L^{2}(M)}$ is finite for some $D>0$.
\end{proposition}

\begin{ex}
\label{ex_setting}
	The most common setting where the hypotheses of Proposition \ref{PROP:prop2} are satisfied is on a sub-Riemannian manifold 
where $\mathbf X$ is a global orthonormal frame for the sub-Riemannian distribution and such that every vector field $X_j$ is divergence free for the chosen measure $\mu$;
an example of such setting is a unimodular Lie group equipped with the Haar measure with a H\"ormander system $\mathbf X$ of left-invariant vector fields.
See also Section \ref{REM:oppImplication}.
\end{ex}

In the statement above, the operator $e^{D\slp^{1/2s}}$ is spectrally defined 
by functional analysis. Indeed, since the sub-Laplacian is essentially self-adjoint on $L^2(M)$, it admits a spectral decomposition:
$$
\slp = \int_{\R} \lambda dE_{\lambda}.
$$
The operator $e^{D\slp^{1/2s}}$ is defined via 
$$
e^{D\slp^{1/2s}} = \int_{\R} e^{D\lambda^{1/s}} dE_{\lambda}, 
$$
and its domain $\Dom{e^{D\slp^{1/2s}}}$ is a subspace of  $L^{2}(M)$; more precisely, it is the set of functions $\phi\in L^2(M)$ such that the following $L^{2}$-norm is finite:
$$
\|e^{D\slp^{1/2s}}\phi\|_{L^{2}(M)}^{2}=
\int_{\R} e^{2D \lambda^{1/2s}} (dE_{\lambda}\phi ,\phi)
<\infty.
$$
Therefore, Proposition \ref{PROP:prop2} may be reformulated as
	$$
	\gamma^s_{\slp} (M)=\cup_{D>0}\Dom{e^{D\slp^{1/2s}}}.
	$$
	This was already mentioned in \cite[Lemma 1.4]{GW1980}.

\begin{proof}[Proof of Proposition \ref{PROP:prop2}]
	Let us consider a function $\phi\in\gamma^s_\slp(M)$. We show that \eqref{eq:def:gs_slp} also holds  for any positive real number $k\in \R^+$. Indeed, given any even integer $a$, by hypothesis, we have
	$$
	\|\slp^{\frac{a}{2}}\phi\|_{L^2}\leq CA^a(a!)^s.
	$$
	Now take any positive real number $b\in\R^+\setminus\N$ and choose  an even integer $a\in\N$ such that $a<b<a+2$. We may write $b:=a\theta+(a+2)(1-\theta)$ with $\theta\in (0,1)$.
	Hence, applying H\"older's inequality to $\|\slp^{\frac{b}{2}}\phi\|_{L^2}$ and using factorial inequalities  in \eqref{ineq_fac} we have 
	\begin{align}\notag
	\|\slp^{\frac{b}{2}}\phi\|_{L^2}&\leq\|\slp^{\frac{a}{2}}\phi\|^\theta_{L^2}\|\slp^{\frac{a+2}{2}}\phi\|^{(1-\theta)}_{L^2}\leq%\big (CA^a(a!)^s\big )^\theta\big (CA^{a+2}((a+2)!)^s\big )^{(1-\theta)}=\\\notag
	%&=
	CA^b(a!)^{s\theta}((a+2)!)^{s(1-\theta)}\leq 2^{3s}C(2^sA)^b(b!)^s.%=C'(A')^b(b!)^s.
	\end{align}
	This shows that \eqref{eq:def:gs_slp} holds for any real number $k>0$.
	Now, using the Taylor expansion for the exponential and \eqref{eq:def:gs_slp} with exponents $k/2s$, we obtain
	\begin{align}\notag
	\|e^{B\slp^{\frac{1}{2s}}}\phi\|_{L^2}&\leq\sum_{k=0}^\infty 
	\frac 1 {k!}\|B^k\slp^{\frac{k}{2s}}\phi\|_{L^2}{\leq}C\sum_k\frac{(BA^{\frac{1}{s}})^k}{k!}\Big(\big(k/s\big)!\Big)^s.
	\end{align}
	By the ratio convergence test for series, using Stirling's approximation, and choosing the constant $B< s A^{-\frac{1}{s}}$, we deduce that the  right-hand side above is finite, as requested.
	
	Conversely, we consider a function $\phi$ such that the $L^2$-norm of $e^{D\slp^{\frac{1}{2s}}}\phi$ is finite for a certain constant $D>0$. Then, taking into account norm properties, we obtain for any integer $k\in\N_0$ and $\phi\in\mathcal C^\infty_0$ that
	\begin{align}\notag
	\|\slp^k \phi\|_{L^2}&%=\|\slp^k e^{-D\slp^{\frac{1}{2s}}}e^{D\slp^{\frac{1}{2s}}}\phi\|_{L^2}
	\leq\|\slp^k e^{-D\slp^{\frac{1}{2s}}}\|_{L^2\rightarrow L^2}\|e^{D\slp^{\frac{1}{2s}}}\phi\|_{L^2}\leq%\\\notag
	%&\leq 
	C\|\slp^k e^{-D\slp^{\frac{1}{2s}}}\|_{L^2\rightarrow L^2}\leq C \sup_{\lambda>0} \lambda^k e^{-D\lambda^{\frac{1}{2s}}}.
	\end{align}
	If we set the function $f(\lambda):=\lambda^k e^{-D\lambda^{\frac{1}{2s}}}$, then its maximum is achieved at $\lambda_D=\big (\frac{2ks}{D}\big)^{2s}$. Therefore, we deduce that
	\[
	\|\slp^k \phi\|_{L^2}%\leq C e^{-2ks}\Big( \frac{2ks}{D}\Big)^{2ks}
	\leq C\Big(\frac{s^s}{D^s}\Big)^{2k}\big((2k)!\big)^s.
	\]
	Defining a constant $A:=s^s/D^s$, this shows $\phi\in\gamma^s_\slp(M)$.	
\end{proof}

\subsection{Lie groups with polynomial volume growth }\label{REM:oppImplication}

In the rest of the paper, we will restrict our attention to manifolds which are Lie groups with polynomial growth of the volume (see e.g. \cite{VSC1992,tER2012} for a definition).

Let us start with showing that 
the hypotheses we added in the statements of Section \ref{subsec:inclusion}
are satisfied in the context of polynomial Lie groups.
Naturally, a Lie group is always equipped with a  Haar measure. In fact, if the Lie group has polynomial growth of the volume, then it is unimodular so left Haar measures are also right-invariant.

\begin{proposition}
\label{prop:genG}
Let $G$ be a  connected Lie group with polynomial growth of the volume.
Let $\mathbf{X} =\{X_1,\dots,X_r\}$ be a H\"ormander system of left-invariant vector fields on $G$ with  associated sub-Laplacian $\slp=-\sum_{j=1}^rX_j^2$.

\begin{enumerate}
\item The operator $\slp$ is 
a non-negative essentially self-adjoint operator on $L^2(M)$, 
and $\mathcal{C}^{\infty}_{0}(G)$ is dense in 
the domain of the self-adjoint extension (for which we keep the same notation $\slp$).
\item 
If $G$ is compact, 
the kernel of $\slp$ is the space of constant functions $\C 1$ and its image is dense in $L^{2}_{0}(G):=(\C 1)^{\perp}$.
If $G$ is non-compact, the self-adjoint extension of $\slp$  is injective on $L^{2}(G)$ and its image is dense in $L^{2}(G)$.
\item 
For $s\geq  0$, the operator $(I+\slp)^{-s/2}$ is a bounded operator on $L^{2}(G)$.
Its convolution kernel $B_{s}$ is square integrable on $G$ for $s$ large enough, i.e. there exists $Q>0$ such that
$B_{s} \in L^{2}(G)$ for $s>Q$.
\item 
The Sobolev embedding holds on $M$ for $\mathbf X$ in the sense of Definition \ref{def:Sob}.
\end{enumerate}
\end{proposition}

\begin{proof}
Parts 1 and 2 as well as the following properties of the volume and of the heat kernel are well known, see \cite{VSC1992,tER2012}. 
Denoting by $V(R)$ 
the volume of the ball about the neutral element $e$ of the group and with radius $R$ for the Carnot-Caratheodory distance associated with $\mathbf{X}$,
the local dimension $d$ and the dimension at infinity $D$ are characterised by:
  $$
\forall R\in (0,1)\qquad C^{-1} R^{d}\leq  V(R) \leq C R^{d}
\qquad\mbox{and}\qquad
\forall R\geq 1 \qquad C^{-1} R^{D}\leq  V(R) \leq C R^{D},
 $$
 for some constant $C>0$.
The heat kernel $h_{t}$, i.e. the convolution kernel of $e^{-t\slp}$,  satisfies:
$$
\forall t>0 \qquad \forall x\in G \qquad  h_t(x^{-1})=h_t(x)\geq 0, 
\qquad
h_t(e)\lesssim V(\sqrt{t})^{{-1}}.
$$ 
Furthermore the mapping $(t,x)\mapsto h_{t}(x)$ is a smooth function on $(0,+\infty)\times G$.
The above properties of the heat kernel implies easily:
$$
\|h_{t}\|_{L^{2}(G)}^{2}=h_{t}*h_{t}(e)=h_{2t}(e)\lesssim V(\sqrt{2t})^{{-1}}.
$$

The properties of the Gamma function (see \eqref{gammafunction}) and the functional calculus of $\slp$ implies 
$$
(I+\slp)^{-\frac s2}
=\frac{1}{\Gamma(\frac s2)}\int_0^\infty t^{\frac s2-1}e^{-t}e^{-t\slp} dt,
$$
so the convolution kernel $B_{s}$ of $(I+\slp)^{-\frac s2}$ satisfies:
$$
\|B_{s}\|_{L^{2}(G)}
\leq \frac{1}{\Gamma(\frac s2)}\int_0^\infty t^{\frac s2-1}e^{-t}\|h_{t}\|_{L^{2}(G)} dt
\lesssim \int_0^\infty t^{\frac s2-1}e^{-t} V(\sqrt{2t})^{{-\frac 12}}dt,
$$
which is finite when $s>d/2$.
This proves Part 3.

Recall that the convolution of two functions  $f,g\in L^{2}(G)$ is a well defined function $f*g$ which is continuous on $G$ and bounded by 
$\|f\|_{L^{2}(G)}\|g\|_{L^{2}(G)}$. 
Therefore, 
for any $\phi\in L^{2}(G)$ in the domain of  $(I+\slp)^{s/2}$ with $s>d/2$, 
we can write
$\phi=B_{s} *  (I+\slp)^{s/2} \phi$ so $\phi$ is continuous  on $G$ and bounded by 
$\|(I+\slp)^{s/2} \phi \|_{L^{2}(G)}\|B_{s}\|_{L^{2}(G)}$.
Now, denoting by $[s/2]$ the smallest integer greater than $s/2$, 
the operator $(I+\slp)^{-[s/2]}(I+\slp)^{s/2}$ is bounded on $L^{2}(G)$ so we have:
$$
\|(I+\slp)^{s/2} \phi\|_{L^{2}(G)}
\leq 
\|(I+\slp)^{[s/2]} \phi\|_{L^{2}(G)}.
$$
Developing the integer powers, we  check easily that 
$$
\|(I+\slp)^{[s/2]} \phi\|_{L^{2}(G)}
\lesssim 
\sum_{j\leq [s/2]}\|\slp^j \phi\|_{L^{2}(G)}
\lesssim 
\sum_{|I|\leq 2[s/2] }\|\mathbf X_I \phi\|_{L^2(G)}
$$
having proceeded as in the proof of Proposition \ref{PROP:prop3}.
This shows Part (4).
\end{proof}

Hence Propositions \ref{PROP:prop1}, \ref{PROP:prop3}, \ref{PROP:prop2}  and \ref{prop:genG} imply easily the following relations between the Gevrey spaces associated with the H\"ormander family $\mathbf X$ and with the corresponding sub-Laplacian
$\slp$:

\begin{cor}
\label{cor:prop:genG}
We continue with the setting of Proposition \ref{prop:genG}.
We have the inclusion:
$$
\gamma^s_{\mathbf{X},L^2}(G) \subset \gamma^s_{\slp}(G) ,
$$
and the equivalences for any function $\phi\in L^{2}(G)$:
$$ 
\phi\in \gamma^s_{\slp}(G)
\ \Longleftrightarrow \ 
\exists D>0 \quad \|e^{D\slp^{1/2s}}\phi\|_{L^{2}(M)}<\infty,
$$
and for any function $\phi$ with compact support:  
$$
\phi\in \gamma^s_{\mathbf{X},L^{\infty}}(G)
\ \Longleftrightarrow \ 
\phi\in  \gamma^s_{\mathbf{X},L^2}(G).
$$
\end{cor}

We continue with the setting of Proposition \ref{prop:genG} and Corollary \ref{cor:prop:genG}.
For the sake of clarity, we assume furthermore that $G$ is  non-compact;
this hypothesis is not necessary but it avoids discussing the cumbersome  technicalities of  $\slp^{-1}$ not being densely defined on $L^2(G)$.
	In this context, to obtain the reverse implication to that in Proposition \ref{PROP:prop3}, we may look at any derivatives of a smooth function $\phi$ in the following way:
	\begin{align}\label{EQ:insertingIdentity}
	\|{\mathbf X}_I\phi\|_{L^2}=\|{\mathbf X}_I(\slp)^{-\frac{|I|}{2}}(\slp)^{\frac{|I|}{2}}\phi\|_{L^2}\leq \|{\mathbf X}_I(\slp)^{-\frac{|I|}{2}}\|_{L^2\rightarrow L^2}\|(\slp)^{\frac{|I|}{2}}\phi\|_{L^2}.
	\end{align}
We are led to study the boundedness of the higher order Riesz transform
	\begin{align}\label{EQ:RieszTransform}
	R_I:={\mathbf X}_I(\slp)^{-\frac{|\alpha|}{2}},
	\end{align}
and the dependence of its $\OP_{L^2}$-norm on $I$.

The Riesz-transform of order 1 are bounded operators with $\OP_{L^2}$-norm $\leq 1$. 
Indeed, since the formal adjoint of $X_{j}$ on $L^{2}(G)$ is $-X_{j}$, we have for $f\in \mathcal{C}^{\infty}_{0}(G)$
	\begin{align*}
	\big(\slp f,f\big)_{L^2}=
	-\sum_{j=1}^k
	\big(X_j^{2} f,f\big)_{L^2}
	=
	\sum_{j=1}^k
	\big(X_j f,X_{j}f\big)_{L^2}
	=
	\sum_{j=1}^k 
	\|X_{j}f\|_{L^{2}}^{2},
	\end{align*}
and so,
	\begin{align}\label{EQ:boundXf}
	\|X_jf\|_{L^2}^2\leq  \big(\slp f,f\big)_{L^2} = \|\slp^{\frac 12} f\|_{{L^{2}}}.   
	\end{align}
For order 2,  it is known \cite{ER1999}  that the higher order Riesz transforms 	$R_I$, $|I|= 2$,  are bounded operators only when $G$ is the local direct product of a connected compact Lie group and a connected nilpotent Lie group. Moreover,
in this case, the transforms $R_{I}$ of all orders are bounded.
However, the proof in \cite{ER1999} does not provide any estimates for their operator norms and their dependence on $I$, and these estimates are needed for our conclusion. 

\section{Subelliptic Gevrey spaces on compact Lie Groups}
\label{SEC:cpt-groups}

In this section, we assume that $G$ is a compact Lie group and we discuss the reverse inclusion to the following one (proved in the previous section, see  Corollary \ref{cor:prop:genG}):
\begin{equation}
\label{eq:reverseinclusion}
	\gamma^s_{\mathbf{X},L^2}(G) \subset \gamma^s_{\slp}(G) .
\end{equation}
While we are still unable to prove it for general compact Lie groups, we will use the well-known non-commutative Fourier analysis on the compact group $SU(2)$ to show the converse inclusion in this case. 
Indeed, the membership in Gevrey spaces in terms of $L^{2}$-norms can be described in terms of the behaviour of Fourier coefficients as a consequence of the Plancherel theorem on $G$. 

We start by setting up the framework for the Fourier analysis on compact Lie groups.

\subsection{Fourier description}

Assume now that $G$ is a \textit{compact} Lie group. 
We equip $G$ with the bi-invariant Haar measure of mass one.
Let $\Gh$ be the unitary dual of $G$, that is, the set of equivalence classes of continuous irreducible unitary representations of $G$. To simplify the notation we will not distinguish between representations and their equivalence classes. Since $G$ is compact, 
$\Gh$ is discrete and all the representations are finite-dimensional. Therefore, given $\xi\in\Gh$ and a basis in the representation space of $\xi$, we can view $\xi$ as a matrix-valued function $\xi:G\rightarrow \C^{d_\xi\times d_\xi}$ where $d_\xi$ is the dimension of this representation space.

For a function $f\in L^1(G)$, the group Fourier transform at $\xi\in\Gh$ is defined as 
\begin{align*}
\widehat{f}(\xi)=\int_G f(x)\xi(x)^* dx,
\end{align*}
where $dx$ is the Haar measure on $G$. Applying the Peter--Weyl theorem (see \cite{RT2009}), we obtain the Fourier inversion formula (for instance for $f\in\mathcal{C}^\infty(G)$)
$$
f(x)=\sum_{\xi\in\Gh} d_\xi {\rm Tr}(\xi(x) \widehat{f}(\xi)).
$$
Moreover, the Plancherel identity holds and we have
\begin{align*}
\|f\|_{L^2(G)}=\Big(\sum_{\xi\in\Gh}d_{\xi}\|\widehat{f}(\xi)\|^2_{HS}\Big)^{1/2}=:\|\widehat{f}\|_{l^2(\Gh)}.
\end{align*}
Here, since $\widehat {f}(\xi)\in\C^{d_\xi\times d_\xi}$ is a matrix, $\|\widehat {f}(\xi)\|_{\HS}$ stands for its Hilbert--Schmidt norm. We recall that for any matrix $A\in\C^{d\times d}$ it is defined by 
\[
\|A\|_{\HS}:=\langle A,A\rangle_{\HS}^{\frac{1}{2}}=\sqrt{\sum_{i,j=1}^d \overline{A_{ij}}A_{ij}}.
\]
Given a left-invariant operator $T$ on $G$ (more precisely $T:\mathcal{D}(G)\rightarrow\mathcal{D}'(G)$ with $T\big(f(x_0\cdot)\big)x=\big(Tf\big)(x_0x)$), its matrix-valued symbol is $\sigma_T(\xi)=\xi(x)^* T\xi(x)\in \C^{d_\xi\times d_\xi}$ for each representation $\xi\in \Gh$. Therefore, formally (or for all $f$ such that $\widehat{f}(\pi)=0$ for all but a finite number of $\pi\in\Gh$) we have
$$
Tf(x)=\sum_{\xi\in\Gh} d_\xi {\rm Tr}(\xi(x)\sigma_T(\xi) \widehat{f}(\xi)).
$$
In other words $T$ is a Fourier multiplier with symbol $\sigma_T$.
For the details of these constructions we refer the reader to \cite{RT2009,RT2013, T1986}.
To simplify the notation, we can also denote $\sigma_T(\xi)$ by $\widehat{T}(\xi)$ or simply by $\widehat{T}$.
For instance, we denote by $\widehat\slp=\widehat\slp(\xi)$ the matrix symbol of $\slp$ at $\xi\in\Gh$. Since $\slp$ is a non-negative operator, it follows that $\widehat{\slp}(\xi)$ is a positive matrix and we can always choose a basis in representation spaces such that $\widehat\slp=\widehat\slp(\xi)$ is a positive diagonal matrix.

\begin{proposition}\label{PROP:ft-equiv}
	Let $G$ be a compact Lie group, 
	let $\mathbf{X}=\{X_1,\dots,X_r\}$ be a H\"ormander system of left-invariant vector fields, and let $\slp=-\sum_{j=1}^rX_j^2 $ be its  associated sub-Laplacian. Let $s>0$.
For any $\phi\in L^{2}(G)$, the following statements are equivalent:
	\begin{enumerate}[(i)]
	\item $\phi\in\gamma^s_\slp(G)$;
	\item There exists  constants $B,K>0$ such that for every $\xi\in\widehat G$ we have
	\begin{equation}\label{HS}
	\|e^{B\widehat\slp(\xi)^{\frac{1}{2s}}}\widehat \phi(\xi)\|_{\HS}\leq K.
	\end{equation}
	\end{enumerate}
\end{proposition} 

\begin{proof}[Proof of Proposition \ref{PROP:ft-equiv}]
	The implication $(i)\Rightarrow (ii)$ is a straightforward consequence of Proposition \ref{PROP:prop2}, the Plancherel identity and the definition of $l^2(\widehat G)$-norm.
Let us show the implication $(ii)\Rightarrow (i)$.
	We assume that there exists $B>0$ such that for every $\xi\in\widehat G$ the estimate \eqref{HS} holds. Take an arbitrary constant $D$ (we will choose it at the end). Then applying the Plancherel identity and the definition of the $l^2(\widehat G)$-norm we have
	\begin{align}\notag
	\|e^{D\slp^{\frac{1}{2s}}}\|^2_{L^2(G)}=\|e^{D\widehat\slp^{\frac{1}{2s}}}\widehat\phi\|^2_{l^2(\widehat G)}=\sum_{[\xi]\in\widehat G} d_\xi\|e^{D\widehat\slp^{\frac{1}{2s}}} \widehat{ \phi}(\xi)\|^2_{\HS}.
	\end{align} 
	Introducing $(I+\widehat\slp)^N(I+\widehat\slp)^{-N}$ with $N\gg 1$ and splitting the exponential we obtain
	\begin{align}\notag
	\|e^{D\slp^{\frac{1}{2s}}}\|^2_{L^2(G)}=\sum_{[\xi]\in\widehat G} d_\xi\|e^{(D-B)\widehat\slp^{\frac{1}{2s}}}(I+\widehat\slp)^N(I+\widehat\slp)^{-N}e^{B\widehat\slp^{\frac{1}{2s}}}\widehat \phi(\xi)\|^2_{\HS}.
	\end{align}
	Now, choose the constant $D$ such that the new constant $D':=D-B$ is strictly less than zero. Then
	\begin{align}\notag
	\|e^{D\slp^{\frac{1}{2s}}}\|^2_{L^2(G)}\leq\sum_{[\xi]\in\widehat G}d_\xi\|(I+\widehat\slp)^{-N}\|^2_{\HS}\|e^{D'\widehat\slp^{\frac{1}{2s}}}(I+\widehat\slp)^N\|^2_{l^2\rightarrow l^2}\|e^{B\widehat\slp^{\frac{1}{2s}}}\widehat\phi(\xi)\|^2_{\HS}.\end{align}
	By hypothesis 
	%we know that there exists a constant $K$ such that 
	$\|e^{B\widehat\slp^{\frac{1}{2s}}}\widehat\phi(\xi)\|^2_{\HS}:= K$ is finite. Furthermore, we can define the multiplier
	\[
	m(\lambda):=e^{D'\lambda^{\frac{1}{2s}}}(I+\lambda)^N,\quad\text{with }D'<0.
	\]
	Formally evaluating this multiplier in $\widehat\slp$ we obtain exactly the operator which we are interested in, that is, $m(\widehat\slp)=e^{D'\widehat\slp^{\frac{1}{2s}}}(I+\widehat\slp)^N$. Thus, we can bound by a constant $K'$ another term in the argument of the previous sum, observing that
	\[
	\|m(\widehat\slp)\|_{\OP_\xi}=\sup_{\lambda\in\sigma(\widehat\slp)}|m(\lambda)|<\infty.
	\]
	Therefore, we obtain
	\[
	\|e^{D\slp^{\frac{1}{2s}}}\|^2_{L^2(G)}\leq K\,K'\sum_{[\xi]\in\widehat G}d_\xi\|(I+\widehat\slp(\xi))^{-N}\|^2_{\HS}.
	\]
The Plancherel formula yields
	$$
	\sum_{[\xi]\in\widehat G}d_\xi\|(I+\widehat\slp(\xi))^{-N}\|^2_{\HS}
	= \|B_{2N}\|_{L^{2}(G)}^{2},
	$$
	which is finite by  Proposition \ref{prop:genG} Part 3 for sufficiently large $N$.
	\end{proof}

\subsection{Elliptic Gevrey spaces on $G$}\label{SSEC:laplacian}

In this Subsection we state the characterisation of the elliptic Gevrey spaces on a compact Lie group $G$, that is, the Gevrey spaces corresponding to the Laplace operator. This was obtained in \cite{DR2014} and here we present an alternative, shorter proof. 

\medskip

We fix a $G$-invariant scalar product on the Lie algebra $\mathfrak{g}$ of the 
 compact Lie group $G$.
 This gives rise to a left-invariant metric on $G$ and the associated Laplace-Beltrami operator is 
 \[
\Delta:=-(X_1^2+\dots+X_n^2).
\]
where $\mathbf{X}=\{X_1,\dots,X_n\}$ is any orthonormal basis on $\mathfrak g$.
In other words, $\Delta$ is the sub-Laplacian associated with the family of (left-invariant) vector fields $\mathbf X$.
Note that if $G$ is simple, there is only one scalar product on $\mathfrak g$ (up to scalar multiplication) and 
$\Delta$ can be identified with the Casimir element of the universal enveloping algebra of $\mathfrak g$. For a general compact Lie group, $\Delta$ can be any (positive linear) combination of the Casimir element of the semi-simple part with the central laplacian.

For all the elements of the unitary dual space of our compact Lie group, $[\xi]=(\xi_{ij})_{1\leq i,j\leq d_{\xi}}\in\widehat G$, we denote by $\lambda^2_{[\xi]}$ the associated eigenvalue for the Laplace--Beltrami operator $\Delta$. Then the eigenvalue corresponding to the representation $[\xi]$ for the operator $(1+\Delta)^{\frac{1}{2}}$ is given by  
\[
\langle\xi\rangle:=(1+\lambda^2_{[\xi]})^{\frac{1}{2}}.
\] 
In accordance with Definition \ref{DEF:gs_slp}, for $\slp=\Delta$ and $\mathbf{X}$ being an orthonormal basis for $\mathfrak{g}$, we can consider the Gevrey spaces $\gamma^s_\Delta(G)$. 
These spaces have been characterised in \cite{DR2014} with arguments  relying on the peculiar properties of the symbolic calculus for the Laplace--Beltrami operator. Here we prove a similar characterisations of $\gamma^s_\Delta(G)$ in terms of Fourier coefficients of functions, and also in terms of the space $\gamma^s_{\mathbf{X},L^{\infty}}(G)$ from Definition \ref{DEF:gsl}, but we develop an alternative, quicker and more elegant argument which does not depend on the symbolic calculus. 

\begin{theorem}\label{THM:DR2014}
Let $G$ be a compact connected Lie group.
Let 
$\mathbf{X}=\{X_1,\dots,X_n\}$ be any orthonormal basis of the Lie algebra $\mathfrak{g}$.
Let $\Delta$ be the Laplace-Beltrami operator.
	Let $0<s<\infty$. The following statements are equivalent:
	\begin{enumerate}
		\item[(i)] $\phi\in \gamma^s_{\mathbf X,L^\infty}(G)$, that is,
	$$	
		\exists A,C>0\qquad \forall a \in \N_0, \qquad \forall I\in T_a^r\qquad 
	\|{\mathbf X}_I \phi\|_{L^\infty(G)}\leq C A^{|I|}(I!)^s;
		$$
		\item[(ii)] $\phi\in \gamma^s_{\mathbf X,L^2}(G)$, that is,
	$$	
		\exists A,C>0\qquad \forall a \in \N_0, \qquad \forall I\in T_a^r\qquad 
	\|{\mathbf X}_I \phi\|_{L^2(G)}\leq C A^{|I|}(I!)^s;
		$$
		\item[(iii)] $\phi\in \gamma^s_\Delta(G)$, that is,
		$$
			\exists A,C>0\qquad \forall b \in \N_0,\qquad
	\|\Delta^b  \phi \|_{L^2(G)} \leq C A^{2b} ((2b)!)^{s},
$$
	or 	equivalently,
	$$
			\exists D>0\qquad \|e^{D\Delta^{\frac{1}{2s}}}\phi\|_{L^2(G)}<\infty;
$$
		\item[(iii)'] there exist $B,K>0$ such that
		$
		\|\widehat {\phi}(\xi)\|_{\HS}\leq K e^{-B\langle\xi\rangle^{\frac{1}{s}}}
		$
		holds for all $\xi\in\widehat G$.
	\end{enumerate}
	
Consequently, 	we have 
$$
	\gamma^s_{\mathbf{X},L^{\infty}}(G)
	\ =\ \gamma^s_{\mathbf{X},L^2}(G)
	\ =\ 	\gamma^s_\Delta(G).
	$$
\end{theorem}

A consequence of this theorem is that the Gevrey spaces thus defined do not depend on the choice of the orthonormal basis $\mathbf X$ for $\mathfrak g$.
However, this point can easily be checked by hand from Defintions 
\ref{DEF:gsl} and \ref{DEF:gslL2} of $\gamma^s_{\mathbf X,L^\infty}(G)$ and $\gamma^s_{\mathbf X,L^2}(G)$.

\begin{proof}[Symbolic-calculus-independent proof of Theorem \ref{THM:DR2014}]	
	The arguments that we have developed so far for subelliptic Gevrey spaces work perfectly in the case of the (Laplace--Beltrami)--Gevrey spaces on compact groups. 
	In fact,  Corollary \ref{cor:prop:genG}	yields the equivalence $(i)\iff (ii)$, 
	and the implication \noindent $(ii)\implies (iii)$, 
	whereas Proposition \ref{PROP:ft-equiv} yields the equivalence $(iii)\iff (iii)'$. It remains to show  $(iii)\implies (ii)$.
Using the equivalence proved in Proposition \ref{PROP:prop2}, combined with the reasoning of Subsection \ref{REM:oppImplication}, the proof of this implication is equivalent to the proof of the boundedness of the higher order Riesz transform ${\mathbf X}_I\Delta^{-\frac{|I|}{2}}$. Note that here we consider operators defined on $L^2_0(G)$, orthogonal complement in $L^2(G)$ of the space of constant functions on $G$, see part $2$ of Proposition \ref{prop:genG}.

 It follows from Subsection \ref{REM:oppImplication} that for every $j\in\{1,\dots,n\}$ we have 
	\begin{align}\label{EQ:order1}
	\|X_j\Delta^{-\frac{1}{2}}\|_{{L^2_0\rightarrow L^2_0}}\leq 1.
	\end{align}
	The commutativity of the Laplace--Beltrami operator plays a fundamental r\^ole to show the boundedness of the Riesz transform for any $\alpha\in\N_0$. In fact, we have
	\begin{align}\notag
	\|{\mathbf X}_I\Delta^{-\frac{|I|}{2}}\|_{L^2_0\rightarrow L^2}=\|X_{i_1}\Delta^{-\frac{1}{2}}X_{i_2}\dots X_{i_{|I|}}\Delta^{-\frac{1}{2}}\|_{L^2_0\rightarrow L^2}\leq 1,
	\end{align}
	obtained applying the inequality \eqref{EQ:order1} repeatedly $|I|$ times. Then we immediately obtain the desidered implication.
	\end{proof}

\subsection{Subelliptic Gevrey spaces on $SU(2)$}
\label{SEC:su2}

In this Subsection we show that in the case of the canonical sub-Laplacian on the special unitary group $SU(2)$, we have the converse inclusion to \eqref{eq:reverseinclusion}:

\begin{theorem} 
\label{thm_{su2}}
We consider the group $G= SU(2)$. Let $\mathbf X=\{X,Y,Z\}$ be a basis for its Lie algebra $\mathfrak{su}(2)$ such that $[X,Y]=Z$, 
and let $\slp:=-(X^2+Y^2)$ be the associated sub-Laplacian.
	Let $0<s<\infty$. The following statements are equivalent:
	\begin{enumerate}
			\item[(i)] $\phi\in \gamma^s_{\mathbf X,L^\infty}(G)$, that is,
	$$	
		\exists A,C>0\qquad \forall a \in \N_0, \qquad \forall I\in T_a^r\qquad 
	\|{\mathbf X}_I \phi\|_{L^\infty(G)}\leq C A^{|I|}(I!)^s;
		$$
		\item[(ii)] $\phi\in \gamma^s_{\mathbf X,L^2}(G)$, that is,
	$$	
		\exists A,C>0\qquad \forall a \in \N_0, \qquad \forall I\in T_a^r\qquad 
	\|{\mathbf X}_I \phi\|_{L^2(G)}\leq C A^{|I|}(I!)^s;
		$$
		\item[(iii)] $\phi\in \gamma^s_\slp(G)$, that is,
		$$
			\exists A,C>0\qquad \forall b \in \N_0,\qquad
	\|\slp^b  \phi \|_{L^2(G)} \leq C A^{2b} ((2b)!)^{s},
$$
	or 	equivalently,
	$$
			\exists D>0\qquad \|e^{D\slp^{\frac{1}{2s}}}\phi\|_{L^2(G)}<\infty;
$$
		\item[(iii)'] there exist $B,K>0$ such that
$\|e^{B\widehat\slp(\xi)^{\frac{1}{2s}}}\widehat \phi(\xi)\|_{\HS}\leq K		$
holds for all $\xi\in\widehat G$.
\end{enumerate}

Consequently, 	we have 
$$
	\gamma^s_{\mathbf{X},L^{\infty}}(SU(2))
	\ =\ \gamma^s_{\mathbf{X},L^2}(SU(2))
	\ =\ 	\gamma^s_\slp(SU(2)).
	$$
\end{theorem}

\begin{proof}[Proof of Theorem \ref{thm_{su2}}]
As in the previous subsection, 
Corollary \ref{cor:prop:genG}	yields the equivalence $(i)\iff (ii)$, 
	and the implication \noindent $(ii)\implies (iii)$, 
	whereas Proposition \ref{PROP:ft-equiv} yields the equivalence $(iii)\iff (iii)'$. So it remains to show  $(iii)\implies (ii)$. 
	Given $I\in T_a^r$, we want to estimate the $L^2$-norm of $\mathbf X_I f$. We follow the same argument as in Subsection \ref{REM:oppImplication}, but we restrict our operators to $L^2_0(SU(2))$. We consider $\mathbf X_I=\mathbf X_I\slp^{-\frac{|\alpha|}{2}}\slp^{\frac{|\alpha|}{2}}$. Then, norm properties and hypotheses yield
	\begin{align*}
	\|\mathbf X_I f\|_{L^2}&\leq\|\mathbf X_I\slp^{-\frac{|\alpha|}{2}}\|_{L^2_0\rightarrow L^2}\|\slp^{\frac{|\alpha|}{2}}f\|_{L^2}\leq\|\mathbf X_I\slp^{-\frac{|\alpha|}{2}}\|_{\OP}CA^{|\alpha|}(|\alpha|!)^s\\
	&\leq\|\mathbf X_I\slp^{-\frac{|\alpha|}{2}}\|_{L^2_0\rightarrow L^2}C(2^sA)^{|\alpha|}(\alpha!)^s,
	\end{align*}
	by \eqref{ineq_fac}.
	Below, we will show that the operator norms are bounded uniformly in $\alpha$ and this concludes the proof.
\end{proof}

The proof of  Theorem \ref{thm_{su2}} will thus be complete once we have shown the following statement which is of interest in its own right.
\begin{proposition}\label{PR:SU2Riesz}
	Let $G=SU(2)$ and assume the hypotheses of Theorem \ref{thm_{su2}}. The higher order Riesz transform $R_I:={\mathbf X}_I\slp^{-\frac{|I|}{2}}$ is bounded on $L^2_0(SU(2))$ with the following operator norm estimate:
	\begin{equation*}
	\exists C>0
	\quad \forall a \in \N,\  I\in T^r_a \qquad 
	\|{\mathbf X}_I\slp^{-\frac{|I|}{2}}\|_{L^2_0\rightarrow L^2} < C^{|I|}.
	\end{equation*}
\end{proposition}

Note that
the boundedness of the higher order Riesz transform $R_I$ on $SU(2)$ is already known, see \cite{ER1999}  and the discussion in Section \ref{REM:oppImplication}.
Our result above shows an estimate for the operator norms of these operators, 
more precisely the fact that they are uniformly bounded with respect to the order $\alpha$.

\begin{remark}
The proof of 	Proposition \ref{PR:SU2Riesz} uses the representation theory of $SU(2)$, or more precisely of its Lie algebra. Hence, it follows that the results in Proposition \ref{PR:SU2Riesz} and  Theorem \ref{thm_{su2}} also hold for the group $G=SO(3)\sim SU(2)/\{\pm I\}$. 
\end{remark}

\subsection{Proof of Proposition \ref{PR:SU2Riesz}}

The proof of Proposition \ref{PR:SU2Riesz} will use the Fourier calculus on $SU(2)$.
The symbols of left-invariant vector fields on $SU(2)$ have been explicitly calculated in 
\cite[Theorem 12.2.1]{RT2009}. They are given by the following formulae
\begin{align}\label{x}
&\sigma_X(l)_{m,n}=-\sqrt{(l-n)(l+n+1)}\delta_{m,n+1}=-\sqrt{(l-m+1)(l+m)}\delta_{m-1,n};\\\label{y}
&\sigma_Y(l)_{m,n}=-\sqrt{(l+n)(l-n+1)}\delta_{m,n-1}=-\sqrt{(l+m+1)(l-m)}\delta_{m+1,n}.
\end{align}
Here we use the customary notation for $SU(2)$, coming from the spin structure, to work with representations $t^l\in\C^{(2l+1)\times (2l+1)}$, $l\in \frac12\N_0$ being half-integers, with components $t^l_{m,n}$, with indices $-l\leq m,n\leq l$ running from $-l$ to $l$ spaced by an integer. Here $\delta_{m,n}$ denotes Kronecker's delta.
The symbol of the sub-Laplacian given by the diagonal matrix whose general entry is
\begin{equation}\label{symbolslp}
\sigma_\slp(l)_{m,n}=(l(l+1)-m^2)\delta_{m,n}.
\end{equation}

A quick glance at these matrices suggests we should consider an equivalent basis for the complexification of $\mathfrak s \mathfrak u_2$ (that is for $\mathfrak s\mathfrak l_2$) whose associated matrix representations have all null entries except on one (upper or lower) diagonal. Thus we define
	\begin{align*}
		&Z:=\frac12 (X-iY),\\
	&\overline Z:=\frac 12(X+iY),
		\end{align*}
	and it can be easily checked that the space $\gamma^s_{L^\infty}(\h_1)$ of functions obtained considering the initial vector fields $\{X,Y\}$ is the same as the one obtained taking into account the elements of the complex basis ${\mathbf Z}:=\{Z,\overline Z\}$. More precisely,
	\begin{multline*}
	\Big\{ f\in \mathcal C^\infty(SU(2))\,|\, \forall a\in\N, I\in T^2_a, \ \|{\mathbf X}_I f\|_{L^\infty}\leq CA^{|I|}(I!)^s,\Big\}
	\\ =
	\Big\{f\in \mathcal C^\infty(SU(2))\,|\, \forall a\in\N, I\in T^2_a, \ \|{\mathbf Z}_I  f\|_{L^\infty}\leq CA^{|I|}(I!)^s\Big\}, 
	\end{multline*}
	where ${\mathbf Z}_I $ is defined in a way similar to ${\mathbf X}_I$, see Convention \ref{conventionI}.
	Therefore, we can reformulate the conclusion as proving the boundedness of the operator ${\mathbf Z} _I \slp^{-\frac{|I|}{2}}$.
We calculate the entries of the matrices associated with the infinitesimal representations of $Z$ and $\overline Z$, obtaining
	\begin{align}\label{mat1}
&\big(\sigma_Z(l)\big)_{m,n}=
	\begin{cases}
	i\frac{\sqrt{(l-n+1)(l+n)}}{2}
	&\text{if }m=n-1\\
	0&\text{otherwise}
		\end{cases},\\
	\label{mat2}	&\big(\sigma_{\overline  Z}(l)\big)_{m,n}=
	\begin{cases}
	i\frac{\sqrt{(l-n)(l+n+1)}}{2}
&\text{if }m=n+1\\
	0&\text{otherwise}
	\end{cases}.
	\end{align}
		Note that the sub-Laplacian is now given by 
	$$
	\slp=2(Z^2 +\overline Z^2).
	$$

We look at the product matrix
	\[
	\sigma_{\mathbf Z_I}=
	\underbrace{\sigma_{Z_{i_1}}(l)\dots\sigma_{Z_{i_{|I|}} (l)}}_{|I|\text{ times}},
	\]
and 	we  observe that each product will produce a matrix with a travelling upper (or lower) non-zero diagonal. Once all the products have been accomplished, the non-zero entries, placed all on one upper (or lower) diagonal,  will be  $\leq  c^{|I|} \big( 1+l^2-m^2\big)^{\frac{|I|}{2}}$, with $c$ a fixed constant. Therefore, the non-zero entries placed all on one upper (or lower) diagonal of the final matrix
	\[\sigma_{Z_{i_1}}(l)\dots\sigma_{Z_{i_{|I|}} (l)}
\sigma_{\slp^{-\frac{|I|}{2}}}(l)
	\]
	will be $\leq 
	c^{|I|} \big( 1+l^2-m^2\big)^{\frac{|I|}{2}}
	\big((l(l+1)-m^2)\big)^\frac{-|I|}{2}\leq {c'}^{|I|}$. The special form of this matrix implies 
		\begin{align*}
	\|\mathbf Z_I \slp^{-\frac{|I|}{2}}\|_{L^2_0\rightarrow L^2}
&=\sup_{l\in \N }\|\sigma_{\mathbf Z_I \slp^{-\frac{|I|}{2}}}(l )\|_{\OP}
=\sup_{l\in \N}\max_{m,n\geq 1}
\Big|\big(\sigma_{\mathbf Z_I \slp^{-\frac{|I|}{2}}}(l )_{m,n}\big)\Big|\\
	&\leq {c'}^{|I|},
	\end{align*} 
	and this concludes the proof of Proposition \ref{PR:SU2Riesz} and therefore also the proof of Theorem \ref{thm_{su2}}.

\section{Subelliptic Gevrey spaces on the Heisenberg Group}
\label{SEC:heis}

The Heisenberg group $\mathbb H_n$ may be described as an important example of a non-abelian (but unimodular) non-compact Lie group. In some sense it is the first stratified nilpotent Lie group.
There is a substantial amount of literature about it and we recall here few titles, such as \cite{FH1987}, \cite{FR2016}, \cite{Fol} and \cite{T1998}.

\subsection{Main result on the Heisenberg group}

Throughout this Section we will look at the Heisenberg group as the manifold $\R^{2n+1}$ endowed with the group law
\[
(x,y,t)(x',y',t'):=(x+x',y+y',t+t'+\frac{1}{2}(xy'-x'y)),
\]
where $(x,y,t),\, (x',y',t')\in\R^n\times\R^n\times\R\sim\h_n$. We consider the canonical basis for the {Heisenberg Lie algebra} $\mathfrak{h_n}$ associated with the Heisenberg group, given by
\begin{align*}
&X_j=\partial_{x_j}-\frac{y_j}{2}\partial_t\quad\text{and}\quad Y_j=\partial_{y_j}+\frac{x_j}{2}\partial_t,\quad\text{for } j\in\{1,\dots,n\},\\
&T=\partial_t.
\end{align*}
These vector fields satisfy the canonical commutation relations 
\[
[X_j,Y_j]=T \quad\text{for every } j\in\{1,\dots,n\},
\] 
with all other possible combinations being zero. This also implies that the set of vector fields  $\mathbf X =\cup_{j=1,\dots,n} \{X_j,Y_j\}$ is a H\"ormander system, with associated sub-Laplacian:
$$
\slp:=-\sum_{j=1}^n(X^2_j+Y_j^2)
$$

The main result of this section is the following description of subelliptic Gevrey spaces:
\begin{theorem} 
\label{thm_heis}
We consider the group $G= \mathbb H_n$. Let $\mathbf X=\cup_{j=1,\dots,n} \{X_j,Y_j\}$ be the H\"ormader system as above and $\slp=-\sum_{j=1}^n(X^2_j+Y_j^2)$ the associated sub-Laplacian.
	Let $0<s<\infty$. Then we have the equality:
$$
\gamma^s_{\mathbf{X},L^2}(G) =\gamma^s_{\slp}(G) ,
$$
and the equivalences for any function $\phi\in L^{2}(G)$:
$$ 
\phi\in \gamma^s_{\slp}(G)
\ \Longleftrightarrow \ 
\exists D>0 \quad \|e^{D\slp^{1/2s}}\phi\|_{L^{2}(G)}<\infty,
$$
and for any function $\phi$ with compact support:  
$$
\phi\in \gamma^s_{\mathbf{X},L^{\infty}}(G)
\ \Longleftrightarrow \ 
\phi\in  \gamma^s_{\mathbf{X},L^2}(G).
$$
\end{theorem}

The  proof of Theorem \ref{thm_heis}
   follows  the same line of arguments as that for $SU(2)$, as well as that in Subsection \ref{REM:oppImplication}. So we will only sketch the ideas:  Corollary \ref{cor:prop:genG} implies that
it suffices to show  the reverse inclusion to \eqref{eq:reverseinclusion},
and this  follows easily from the uniform boundedness of higher order Riesz transform, that is, the property in the following statement: 

\begin{proposition}\label{PR:HnRiesz}
As above, 
we consider the $\slp=-\sum_{j=1}^n(X^2_j+Y_j^2)$ on the Heisenberg group $G= \mathbb H_n$. 

The higher order Riesz transform $R_I={\mathbf X}_I \slp^{-\frac{|I|}{2}}$ is bounded on $L^2(\h_1)$ with the following operator norm estimate:
	\begin{equation*}
	\exists C>0
	\quad \forall a \in \N,\  I\in T^r_a \qquad 
	\|{\mathbf X}_I\slp^{-\frac{|I|}{2}}\|_{L^2\rightarrow L^2} \leq C^{| I |}.
	\end{equation*}
\end{proposition}	

This result is analogous to the one on $SU(2)$ given in Proposition \ref{PR:SU2Riesz} and the same remark applies:
although the higher order Riesz transform $R_\alpha$ on $\mathbb{H}_n$ is already known to be bounded, see \cite{ER1999}  and the discussion in Section \ref{REM:oppImplication}, 
our result above shows an estimate for the operator norms of these operators.

\subsection{Fourier description}

For each $\lambda\in\R\setminus\{0\}$, the corresponding Schr\"odinger representation 
\[
\pi_\lambda:\mathbb{H}_n\rightarrow\mathcal{U}(L^2(\Rn)),
\]
is a unitary irreducible representation given by 
\begin{align}\notag
\pi_\lambda(x,y,t)\phi(u)=[\pi_1(\sqrt\lambda x,\sqrt\lambda y,\lambda t)](u)=e^{i\lambda(t+\frac{1}{2}xy)}e^{i\sqrt\lambda yu}\phi(u+\sqrt{|\lambda|}x).
\end{align}
In the above definition we use the following convention from \cite{FR2016}:
\begin{align*}
\sqrt\lambda:={\rm sgn}(\lambda)\sqrt{|\lambda|}=
\begin{cases}
\sqrt\lambda&\text{if }\lambda>0,\\
-\sqrt{|\lambda|}&\text{if }\lambda<0.
\end{cases}
\end{align*}

We move now to the infinitesimal representations associated to the Schr\"odinger representations. They play a crucial r\^ole in determining the symbols of left-invariant differential operators. Considering the aforementioned canonical basis of $\mathfrak h_n$, for every $\lambda \in\R\setminus\{0\}$ the corresponding infinitesimal representations of the elements of the basis are given by
\begin{subequations}
	\label{infrep}
	\begin{align}
	&\pi_\lambda(X_j)=\sqrt{|\lambda|}\partial_{x_j}\quad&\text{ for }j=1,\dots,n;\label{eq1}\\
	&\pi_\lambda(Y_j)=i\sqrt\lambda x_j \quad&\text{ for }j=1,\dots,n;\label{eq2}\\
	&\pi_\lambda(T)=i\lambda I.\label{eq3}
	\end{align}
\end{subequations}
We recall that for every $\lambda\in\Rn\setminus\{0\}$ the space of all smooth vectors $\mathcal H^\infty_{\pi_\lambda}$ is the Schwartz space $\mathcal{S}(\Rn)$. An easy calculation yields that the infinitesimal representation of the sub-Laplacian $\slp$ is given by 
\begin{align}\label{irslp}
\pi_\lambda(\slp)=|\lambda|\sum_{j=1}^n(x_j^2-\partial_{x_j}^2),
\end{align}
which is clearly related to the harmonic oscillator 
\begin{align*}
H=-\Delta + |x|^2. 
\end{align*}

We now recall the matrix representation of the operators 
\eqref{infrep} and \eqref{irslp}.
To simplify the notation, we will work with the three-dimensional Heisenberg group $\mathbb{H}_1$, i.e. $n=1$. The extension to any $n$ is straightforward. It is well known that the \textit{Hermite polynomials}, once normalised, form an orthonormal basis of $L^2(\R)$ consisting of eigenfunctions of $\pi_\lambda(\slp)$, (see \cite{Fol} and \cite{J2014} for two different proofs). \\% can be found. The former relies on complex function theory and the latter on real function descriptions of the fundamental theorem of calculus and the Schwarz inequality.}. %Here we will fix the notation and recall some properties of these polynomials, see e.g. \cite{S1975} for more details.  

For every $k\in\N$ and $x\in\R$ the ($k$-th)-Hermite polynomial is given by
\[
H_k(x):=(-1)^ke^{x^2}\Big(\frac{d^k}{dx^k}e^{-x^2}\Big),
\]
%We see that
%\[
%\int_\R{e^{-x^2}H_k(x)H_m(x)}dx=\pi^{\frac{1}{2}}2^{k}k!\delta_{km},
%\]
and the {normalised Hermite functions} are defined by
\begin{align}\label{h}
h_k(x):=\frac{1}{\sqrt{\sqrt\pi 2^kk!}}e^{-\frac{x^2}{2}}H_k(x)=c_ke^{-\frac{x^2}{2}}H_k(x),
\end{align}
where $c_k:=\frac{1}{\sqrt{\sqrt\pi 2^kk!}}$. %As mentioned above, the normalised Hermite functions $\{h_k(\cdot)\}_{k\in\N_0}$ form a basis of $L^2(\R)$.
%\begin{remark}
%	For the sake of completeness we can observe that the higher dimensional Hermite functions are products of one-dimensional Hermite functions, namely for every multi-index $\alpha=(\alpha_1,\dots,\alpha_n)$ and $x\in\Rn$ we have $h_\alpha(x)=h_{\alpha_1}(x_1)\dots h_{\alpha_n}(x_n)$. The family $\{h_{\alpha}\}_{\alpha\in\N^n_0}$ provides an orthonormal basis for $L^2(\Rn)$.
%\end{remark}
The well-known properties of Hermite polynomials (see e.g. \cite{S1975}) allow us to calculate the matrices corresponding to the infinitesimal representations of the elements of the fixed canonical basis of $\mathbb H_1$, i.e. $X$ and $Y$, and of the sub-Laplacian. Therefore, with our notation, we have:
%In order to do this, we recall from \cite{S1975} useful properties of the Hermite functions: for all $k\in\N, \,k\geq2$ and $x\in\R$ we have
%\begin{subequations}
%	\label{prop}
%	\begin{align}
%	&H_k(x)=2 x H_{k-1}(x)-2(k-1)H_{k-2}(x),\label{prop1}\\
%	&H'_k(x)=2kH_{k-1}(x)\label{prop2}.
%	\end{align}
%\end{subequations}
%Equations \eqref{eq2} and \eqref{prop1} imply the following equality
%\[
%\pi_\lambda (Y)H_k(x)=i\sqrt\lambda x H_k(x)=i\sqrt\lambda\Big(\frac{1}{2}H_{k+1}(x)+kH_{k-1}(x)\Big).
%\]
%Multiplying both sides by $c_ke^{-\frac{x^2}{2}}$ and looking at the definition of \eqref{h}, we immediately obtain
\begin{align*}%\label{mol}
\pi_\lambda(Y)h_k(x)&=i\sqrt{\lambda}\Bigg(\sqrt{\frac{k+1}{2}}h_{k+1}(x)+\sqrt\frac{k}{2}h_{k-1}(x)\Bigg),\\
\pi_{\lambda}(X)h_k(x)&=-\sqrt{|\lambda|}\sqrt\frac{k+1}{2}h_{k+1}(x)+\sqrt{|\lambda|}\sqrt{\frac{k}{2}}h_{k-1}(x),\\%\label{EQ:der}\\
\pi_{\lambda}(\slp)h_k(x)&=|\lambda|(2k+1)h_k(x).\notag
\end{align*}
We use the same notation $\pi_\lambda(X)$, $\pi_\lambda(Y)$ and $\pi_\lambda(\slp)$ to denote both the operators and the infinite matrices associated to our vector fields with respect to the orthonormal basis comprising the Hermite functions $\{h_k\}_{k\in\N}$. Then for all $k,l\in\N$ the $(k,l)$-entries of these matrices are given by
\begin{align}\label{matrix1}
&\big(\pi_\lambda(\slp)\big)_{k,l}=|\lambda|(2k+1)\delta_{k,l},\\\label{matrix2}
&\big(\pi_\lambda(X)\big)_{k,l}=
\begin{cases}
\sqrt{|\lambda|}\sqrt{\frac{k+1}{2}}&\text{if }k=l-1\\
-\sqrt{|\lambda|}\sqrt{\frac{k}{2}}&\text{if }k=l+1\\
0&\text{otherwise}
\end{cases},\\\label{matrix3}
&\big(\pi_\lambda(Y)\big)_{k,l}=
\begin{cases}
i\sqrt{\lambda}\sqrt{\frac{k+1}{2}}&\text{if }k=l-1\\
i\sqrt{\lambda}\sqrt{\frac{k}{2}}&\text{if }k=l+1\\
0&\text{otherwise}
\end{cases}.
\end{align}

\subsection{Proof of Proposition \ref{PR:HnRiesz}}

	To simplify the notation, once again, we will work with the three-dimensional case, i.e., $n=1$.  Formulae \eqref{matrix2} and \eqref{matrix3} provide explicit expressions for the entries of the matrices associated with the vector fields of the elements of the canonical basis of $\mathfrak h_1$. 
	As in the case of $SU(2)$, 
	a quick glance at these matrices suggests we should consider an equivalent basis for the complexification of $\mathfrak h_1$ whose associated matrix representations have all null entries except on one (upper or lower) diagonal. Thus we define
	\begin{align*}
	&Z:=\frac12 (X-iY),\\
	&\overline Z:=\frac 12(X+iY),\\
	&T=\frac{2}{i}[Z,\overline Z],
	\end{align*}
	and it can be easily checked that the space $\gamma^s_{L^\infty}(\h_1)$ of functions obtained considering the initial family of vector fields $\mathbf X= \{X,Y\}$ is the same as the one obtained taking into account the elements of the complex basis $\mathbf Z= \{Z,\overline Z\}$. More precisely,
	\begin{multline*}
	\Big\{ f\in \mathcal L^2(\h_1)\,|\, \forall a\in\N, I\in T^2_a, \ \|{\mathbf X}_I f\|_{L^2}\leq CA^{|I|}(I!)^s,\Big\}
	\\ =
	\Big\{f\in \mathcal L^2(\h_1)\,|\, \forall a\in\N, I\in T^2_a, \ \|{\mathbf Z}_I  f\|_{L^2}\leq CA^{|I|}(I!)^s\Big\}, 
	\end{multline*}
	where ${\mathbf Z}_I $ is defined in a way similar to ${\mathbf X}_I$, see Convention \ref{conventionI}.
	Therefore, we can reformulate the conclusion as proving the boundedness of the operator $\mathbf Z_I \slp^{-\frac{|I|}{2}}$.
	Without loss of generality we can restrict to the case $\lambda>0$. We calculate the entries of the matrices associated with the infinitesimal representations of $Z$ and $\overline Z$, obtaining
	\begin{align}\label{mat1}
&\big(\pi_\lambda(Z)\big)_{k,l}=
	\begin{cases}
	\sqrt{\lambda}\sqrt{\frac{k+1}{2}}&\text{if }k=l-1\\
	0&\text{otherwise}
		\end{cases},\\
	\label{mat2}	&\big(\pi_\lambda(\overline  Z)\big)_{k,l}=
	\begin{cases}
	-\sqrt{\lambda}\sqrt{\frac{k}{2}}&\text{if }k=l+1\\
	0&\text{otherwise}
	\end{cases}.
	\end{align}
	Note that the sub-Laplacian is now given by 
	$$
	\slp=2(Z^2 +\overline Z^2).
	$$
	
	We observe that since all the vector fields are left-invariant, for all $\lambda\in\R\setminus \{0\}$ we have
	\[
	\pi_{\mathbf Z_I\slp^{-\frac{|I|}{2}}}(\lambda)=\pi_{Z_{i_1}\dots Z_{i_{|\alpha|}}\slp^{-\frac{|I|}{2}}}(\lambda)=\pi_{Z_{i_1}}(\lambda)\dots\pi_{Z_{i_{|I|}}}(\lambda)\pi_{\slp^{-\frac{|I|}{2}}}(\lambda).
	\]
	Therefore, we can evaluate the non-zero entries of the matrix product above in order to estimate the norm of our operator.
	Summarising, according to \eqref{mat1} and \eqref{mat2}, we have the following:
	\begin{enumerate}[1.]
		%\item $Z_1\slp^{-\frac{1}{2}}$ is a bounded operator, thus $\|Z_1\slp^{-\frac{1}{2}}\|_{op}\leq 1$;
		\item $\pi_{Z_j}(\lambda)$ (or $\pi_{\overline Z_j}(\lambda)$) is a matrix whose entries equal zero except on the first upper (or lower) diagonal. Precisely, they are of the form $c_1{\big(\lambda m\big)^\frac{1}{2}}$, where $c_1$ is a fixed constant;
		\item $\pi_{\slp^{-\frac{|I|}{2}}}(\lambda)$ is a diagonal matrix whose diagonal entries are of the form ${c_2 \big(\lambda m\big)^\frac{-|I|}{2}}$, where $c_2$ is a fixed constant.
	\end{enumerate}
	If we look at the product matrix
	\[
	\underbrace{\pi_{Z_{i_1}}(\lambda)\dots\pi_{Z_{i_{|I|}}}(\lambda)}_{|I|\text{ times}}
	\]
	we can observe that each product will produce a matrix with a travelling upper (or lower) non-zero diagonal. Once all the products have been accomplished, the non-zero entries, placed all on one upper (or lower) diagonal,  will be of the form $ c^{|I|}\big(\lambda m\big)^{\frac{|I|}{2}}$, with $c$ a fixed constant. Therefore, the non-zero entries placed all on one upper (or lower) diagonal of the final matrix
	\[
	\pi_{Z_{i_1}}(\lambda)\dots\pi_{Z_{i_{|I|}}}(\lambda)
	\pi_{\slp^{-\frac{|I|}{2}}}(\lambda)
	\]
	will be of the form $c^{|I|}\big(\lambda m\big)^\frac{|I|}{2}\big(\lambda m\big)^\frac{-|I|}{2}\sim c^{|I|}$. 
	The special form of this infinite dimensional matrix implies 
		\begin{align*}
	\|\mathbf Z_I \slp^{-\frac{|I|}{2}}\|_{\OP}&=\sup_{\lambda\in\R\setminus \{0\}}\|\pi_{\mathbf Z^I\slp^{-\frac{|I|}{2}}}(\lambda)\|_{\OP_{\lambda}}=\sup_{\lambda\in\R\setminus \{0\}}\max_{m,k\geq 1}\Big|\big(\pi_{\mathbf Z^I\slp^{-\frac{|I|}{2}}}(\lambda)_{m,k}\big)\Big|\\
	&\leq c^{|I|},
	\end{align*} 
	and this concludes the proof of Proposition \ref{PR:HnRiesz} and therefore also the proof of Theorem \ref{thm_heis}.
	
\begin{remark}
The proof of Propositions \ref{PR:SU2Riesz} and \ref{PR:HnRiesz} 
about the operator bound of the higher Riesz transforms on $SU(2)$ and $\h_1$ are virtually the same. We think that the result on $SU(2)$	implies the result on $\h_1$ by contraction, see \cite{fulvio}.
\end{remark}

\end{document}